\documentclass[12pt]{amsart}
\usepackage{amsmath,amsfonts,amssymb,amsthm, esint}
\usepackage[english]{babel}
\usepackage{enumerate}
\usepackage{dsfont}
\usepackage{comment}
\usepackage[colorlinks,citecolor=blue]{hyperref}
\usepackage{epsfig,graphicx} 
\usepackage{latexsym, amsxtra, mathrsfs, bm}
\usepackage{graphics}
\usepackage{verbatim}
\usepackage[abs]{overpic}
\usepackage{mathtools}
\usepackage{amscd}
\allowdisplaybreaks[3]
\theoremstyle{plain}
        \newtheorem{theorem}{Theorem}[section]
        \newtheorem*{theorem*}{Theorem}
        \newtheorem*{conj*}{Conjecture}
        \newtheorem{lemma}[theorem]{Lemma}
        \newtheorem{prop}[theorem]{Proposition}
        \newtheorem{corollary}[theorem]{Corollary}

\theoremstyle{definition}
        \newtheorem{definition}[theorem]{Definition}
        \newtheorem{rem}[theorem]{Remark}

\theoremstyle{remark}
        \newtheorem*{remark}{Remark}

\numberwithin{equation}{section}
\numberwithin{theorem}{section}
\numberwithin{table}{section}
\numberwithin{figure}{section}

\begin{document}
\title{Bifurcations and invariant sets for a family of replicator maps from evolutionary games}
\author{Sergey Kryzhevich}
\address[Sergey Kryzhevich]{
	Institute of Applied Mathematics, Faculty of Applied Physics and Mathematics, Gda\'nsk University of Technology, 80-233, Gda\'nsk, Poland}
\email[Sergey Kryzhevich]{sergey.kryzhevich@pg.edu.pl}
\author{Yiwei Zhang}
\address[Yiwei Zhang]{School of Mathematical Sciences and Big Data, Anhui University of Science and Technology, Huainan, Anhui 232001, P.R. China}
\email[Yiwei Zhang]{yiweizhang831129@gmail.com}
\author{Magdalena Chmara}
\address[Magdalena Chmara]{Institute of Applied Mathematics, Faculty of Applied Physics and Mathematics, Gda\'nsk University of Technology, 80-233, Gda\'nsk, Poland	}
\email[Magdalena Chmara]{magdalena.chmara@pg.edu.pl}

\thanks{}

\begin{abstract} We study the dynamics of a family of replicator maps, depending on two parameters. Such studies are motivated by the analysis of the dynamics of evolutionary games under selections.  From the dynamics viewpoint, we prove the existence of hyperbolic chaos for the considered map. Moreover, we also give a partial solution of an open problem formulated in \cite{Misiurewicz1}: to describe all the one-dimensional maps with all the periodic orbits having the same mean value.  
\end{abstract}


\maketitle

\section{Introduction}
\subsection{Background}
The natural world is governed by complex processes involving numerous interacting organisms, which can be observed across various mass scales, from unicellular life forms to complex multicellular eukaryotes. These interactions also manifest at different organizational levels, involving multiple species or different inheritable traits within a single species, spanning ecological, epidemiological, and immunological contexts.

A compelling example of this is the formation of ecosystems through community interactions. Research \cite{hastings:88:foodweb,nee:90:commcon} has confirmed observations from field and laboratory studies \cite{weiss:23:gutcomm} on intricate relationships within a complex network of diverse components that adapt their compositions to environmental changes. These interactions may also involve competition for shared resources or environments.

A key theme in the study of these evolutionary systems is the connection between local behavioral interactions of members and the dynamics of changes in population composition over extended periods. These population dynamics may either maintain equilibrium or cycle indefinitely \cite{nowak:04:popdynecol}. Examples include general epidemiological contexts \cite{anderson:91:infechum} where the focus is on hosts and parasite populations similar to predator-prey systems, immunological contexts on the coinfection of multiple parasite strains within a single host \cite{nowak:94:supinf}, and even animal behavior settings \cite{kerr:02:localdis,liao:19:rpseng,sinervo:96:rpsmale}.

Dynamical systems methods \cite{hastings:93:chaosecol,may:74:biopop} are particularly valuable as they provide qualitative insights and sometimes reveal complex population dynamics that experts can understand and validate through controlled laboratory experiments \cite{constantino:97:chaosinsect}. Quantitative analysis tools \cite{rogers:22:chaosnature,toker:20:detectchaos} can be developed and applied to extensive field observations data sets to identify patterns that reflect the mechanisms hypothesized driving these systems.

Central to the study of population biology are the interactions between natural selection and environmental factors and their collective impact on population dynamics. Both biotic and abiotic elements can independently influence these dynamics by altering the selective pressure, as evidenced by ecological field research \cite{saccheri:06:selpopdyn}. Population cycles have been documented in such studies \cite{sinervo:00:cyclesoff,yoshida:03:evoecodyn}, leading to inquiries about the potential for more intricate dynamics under varying selection pressures \cite{benton:09:abiotic}.

Given the scarcity and difficulty of obtaining robust empirical data from field studies \cite{beninca:08:chaosplankton,beninca:15:spefluc}, evolutionary game dynamics \cite{hofbauer:03:egd}—a synthesis of biological theory with game and dynamical theories—can help bridge these knowledge gaps. Typically, this involves analyzing specific dynamical systems that accurately represent the key mechanisms driving the dynamics of interest \cite{beninca:15:spefluc,chen:17:evonplayerhd,doebeli:14:chaosevo,ebenman:96:evopopdyn}. Quantitative measures, such as positive Lyapunov exponents derived from numerical methods and informed by a mathematical understanding of these systems, can confirm the presence of complex (and possibly chaotic) dynamics, even when accounting for stochastic elements that are part of the deterministic framework \cite{constantino:97:chaosinsect,hastings:93:chaosecol}. 

In this paper, we follow the approach and arguments initially presented by May \cite{may:74:biopop,may:76:modelcomp}. We examine a broad class of one-dimensional dynamical systems that maintain frequency-dependent reproduction for two heritable traits and include a parameter that modulates selection intensity. Despite the increased complexity of the mathematical model (for example, the difference function is now transcendental rather than quadratic), we will show that it has similar dynamical properties and can be analyzed using many analytical tools \cite{demelo:93:1ddynamics,vanstrien:10:1ddynamicsnew} that were previously only applied to single-population dynamical systems.

As a preliminary demonstration, we identify inherent properties within this family that explain the emergence of complex population dynamics as selection pressure increases. Furthermore, it provides a framework where the qualitative effects of natural selection on populations can be directly connected to well-understood one-dimensional dynamical systems. There is potential for profound mathematical insights to have corresponding natural interpretations for phenomena observed in biological populations.

\subsection{A dynamical model}
In this subsection, we mathematically describe our dynamical model and how it is derived from two-player matrix games.

\subsubsection{Evolutionary Games}\label{subsection:evolutionarygames}
In the context of Evolutionary Games, a population that engages in interactions modeled as a two-player symmetric noncooperative game with a finite number of pure strategies will reach a unique state. This state corresponds to one of the solutions of the game. The population state, denoted as $\mathbf{p}$, is characterized by fixed proportions of individuals using specific strategies, akin to the selection probabilities when the entire population acts as a mixed strategy player in the best response to itself, resulting in a Nash equilibrium $(\mathbf{p},\mathbf{p})$. In this equilibrium, all strategies have equal total payoffs, so no individual would benefit from switching to a different strategy. 

For games with two strategies ${1,2}$, the solution can be relatively easy to derive. The $2 \times 2$ payoff matrix $\mathbf{A} = (a_{ij} : i,j = 1,2)$ defines the payoff $a_{ij}$ for the outcome of the game between the pair $(i,j)$, where $i,j \in {1,2}$. The state vector $\mathbf{p} = (p_1,p_2)^T$ represents the proportions $p_1,p_2 \in [0,1]$ of the two strategies, and the vector $\mathbf{w} = (w_1,w_2)^T = \mathbf{A} \mathbf{p}$ accumulates their total payoffs. By setting $p_1 = p$ and $p_2 = 1 - p$ and solving the linear equations
\begin{align}
w_1 &= pa_{11} + (1 - p)a_{12} \\
w_2 &= pa_{21} + (1 - p)a_{22}
\end{align}
with the condition $w_1 = w_2$, we obtain the Nash equilibrium solution $p_{\mathrm{NE}} = (a_{22} - a_{12})/(a_{11} - a_{21} + a_{22} - a_{12})$. For the Hawks and Doves evolutionary game with the payoff matrix
\begin{align}\label{HawksandDoves}
\mathbf{A}:=
\left[ \begin{array}{cc}
(V - C)/2 & V \\
0 & V/2
\end{array} \right]
\end{align}
where $V$ and $C$ are positive, and $V < C$, the population equilibrium is at $p_{\mathrm{NE}} = V/C$.

\subsubsection{A family of replicator maps}
When the replicator incorporates the effects of external factors that collectively mediate the strength of selection $e^{{\gamma}w_i}$. Two-player matrix games mentioned in Subsection \ref{subsection:evolutionarygames} naturally introduce a \emph{replicator map} $f:[0,1]\to [0,1]$ by
\begin{align}\label{genpropreplicator1}
f(p) &= \frac{p \cdot e^{{\gamma}(pa_{11} - pa_{12} + a_{12})}}{p \cdot e^{{\gamma}(pa_{11} - pa_{12} + a_{12})} + (1 - p) \cdot e^{{\gamma}(pa_{21} - pa_{22} + a_{22})}}.
\end{align}
Here, $p \in [0,1]$ represents the proportion of one of the two strategies adopted in the population since $p_{2}=1-p_{1}$. Each entry $a_{ij}$ in the game matrix $\mathbf{A}$ accounts for the gains in fitness an individual of trait-$i$ has against one of trait-$j$ from a single interaction. 

The inverse temperature parameter $\gamma := 1/T$ (see, e.g. \cite{kirkpatrick:83:simann}) has a biological interpretation to modulate the frequency-reproduction mechanism in a population with many-body interactions. Low-valued $\gamma$ (resp. high valued $T$) settings are associated with reproductions predominantly governed by fitness gains due to trait differences $\mathbf{A}$ amidst some negligible effects from external factors. This could change and is reflected by higher-valued $\gamma$ (lower-valued $T$) settings that act to reduce those trait differences, and the overall effect on reproduction would appear a matter of individuals' circumstances.

For Hawk-Dove interactions, algebraic simplification of \eqref{HawksandDoves} and \eqref{genpropreplicator1} results in a \emph{three-parameter $(\gamma,V,C)$ family of replicator maps,} \begin{equation}\label{equ:Boltzmann}
f_{C,V,\gamma}(x)=\frac{x}{x+(1-x) e^{\frac{C\gamma}2 \left(x-\frac{V}{C}\right)}}
\end{equation}
with parameters $C, V, \gamma>0$ and $C>V$.

The rest of the paper is organized as follows. In Section 2, we discuss the properties of the considered family of maps (mainly obtained by others). In Section 3, we list what is known about the periodic solutions and bifurcations. In Sections 4 and 5 we formulate and prove a new result on the existence of a hyperbolic chaotic repeller of the considered system. Finally, in Section 4 we discuss an interesting fact about replicator mappings: the mean values of all periodic solutions (except edges 0 and 1) coincide. We introduce a broad class of mappings posessing the same property and thus partially solve an open problem, formulated in \cite{Misiurewicz1}.

\section{Basic properties}

In this section, we will discuss some basic properties of the family of replicator maps \eqref{equ:Boltzmann}. These properties will be useful in the sequel.

\subsection{Reparametrization} 
It is easy to see that, in fact, the replicator map depends on only two parameters: $a:=c\gamma/2$ and $b:=v/c$. So, we rewrite \eqref{equ:Boltzmann} in the following form\footnote{In \cite{Misiurewicz4} and \cite{Misiurewicz1}, $f$ is named by $f_{MW}$. Here MW refers to the study of the evolution in a simple population game where agents are using multiplication weights algorithm (cf. \cite{AHK12}).}: 
\begin{equation}\label{equ:Boltzmann0}
f(x)=f_{a,b}(x)=\frac{x}{x+(1-x) e^{a\left(x-b \right)}}
\end{equation}
with parameters $a>0$ and $b\in (0,1)$. See Figure \ref{fig_1} for the sample graph of $f_{a,b}$. 

\begin{figure}[t!]\centering
  \includegraphics[width=.84\textwidth]{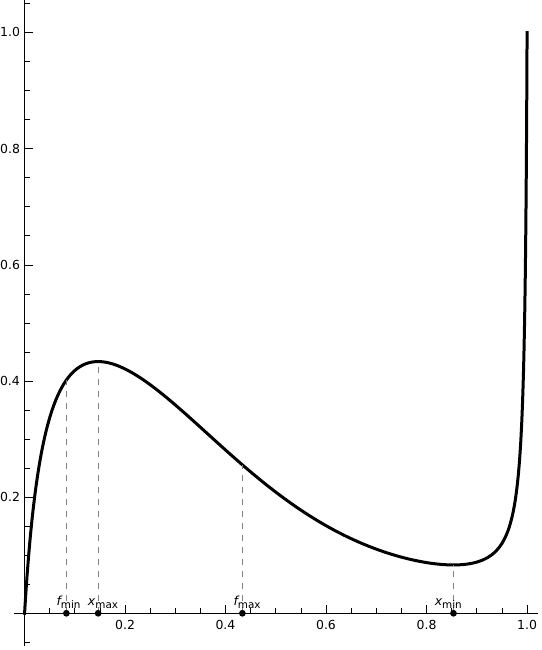}
\caption{\footnotesize The map $f_{8,1/3}$.}
\label{fig_1}
\end{figure}

\subsection{Fixed points}

Observe that the map $f$ has exactly three fixed points $0$, $b$, and $1$. To study their stability, let us calculate 
\begin{equation}\label{eq0}
f'_{a,b}(x)=\dfrac{e^{a\left(x-b\right)}\left(a x^2-a x+1\right)
}{\left(x+(1-x)e^{a\left(x-b\right)}\right)^2}.
\end{equation}
Consequently, 
$$f'_{a,b}(0)=e^{a b}>1, \qquad f'_{a,b}(1)=e^{a (1-b)}>1.$$
Evidently, both points $0$ and $1$ are unstable. 

Meanwhile
$$f'_{a,b}\left(b\right)=1-a b(1-b)<1$$
Consequently, the point $b$ is stable if $f'(b)>-1$, that is
$$a<\dfrac{2}{b(1-b)},$$ 
and unstable if
$$a>\dfrac{2}{b(1-b)}.$$ 

In addition, from Equation \eqref{eq0} it follows that the map $f_{a,b}$ is monotone if $a\le 4$ (so its dynamics is trivial). In what follows, we always assume that \begin{equation}\label{eq1}
a>4.
\end{equation}

Moreover, in \cite[Theorem 3.8]{Misiurewicz1} it was shown that the fixed point $b$ attracts all the points of the interval $(0,1)$ provided in $a\le 2/b(1-b)$. 

\subsection{Recursive formula}

It follows from \cite[Equation (13)]{Misiurewicz1} that the $n$-th iteration of our map $f_{a,b}$ is given by the formula:
$$f_{a,b}^n(x)=\frac{x}{x+(1-x) e^{a\left(\sum_{i=0}^{n-1}\left(f_{a,b}^i(x)-b\right)\right)}},\quad \forall x \in[0,1], \quad \forall n \in {\mathbb N}.$$

\subsection{Critical points}

Let us observe, first of all, that  if \eqref{eq1} is true, the mapping 
$f_{a,b}$ has exactly two critical points:
$$
x_{max} = \dfrac12-\sqrt{\dfrac14-\dfrac1{a}}
\quad\mbox{and}\quad 
x_{min} = \dfrac12+\sqrt{\dfrac14-\dfrac1{a}}.
$$
Evidently, $x_{min}+x_{max}=1$ and $x_{min} x_{max}=1/a$. This gives us the following asymptotic estimates for those values: 
$$
x_{max}=\dfrac1{a}+o\left(\dfrac1{a}\right), \qquad 
x_{min}=1-\dfrac1{a}+o\left(\dfrac1{a}\right).
$$
Besides, we can estimate
$$
f_{min}:=f_{a,b}(x_{min})=\dfrac{1}{1+(1/x_{min}-1)e^{a(x_{min}-b)}}=e^{a(b-1)+o(a)}
$$
as $a\to +\infty$.
Similarly, 
$$
f_{max}:=f_{a,b}(x_{max})=\dfrac{1}{1+(1/x_{max}-1)e^{a(x_{max}-b)}}=1-e^{-ab+o(a)}
$$
as $a\to +\infty$.

\subsection{Symmetry} 

\begin{lemma}\label{lemma1} \cite[Equation (10)]{Misiurewicz1}.
For any $x,b\in [0,1]$ and $a>0$, we have 
$$f_{a,1-b}(x)=1-f_{a,b}(1-x).$$
\end{lemma}

This means that the transformation $x\mapsto 1-x$ converts the map $f_{a,b}$ to $f_{a,1-b}$. 
In particular,
$$f^n_{a,1-b}(x)=1-f^n_{a,b}(1-x)$$
for any $n\in {\mathbb N}$. So, it suffices to study the case $b\le 1/2$ only. 

\subsection{Schwartzian derivative}

Recall the notion of the Schwartzian derivative:
$$Sf_{a,b}(x)=\left(\dfrac{f''_a(x)}{f'_{a,b}(x)}\right)'-\dfrac12
\left(\dfrac{f''_a(x)}{f'_{a,b}(x)}\right)^2.$$

The following technical statement is proven in \cite[Proposition 3.2]{Misiurewicz1}.

\begin{lemma}\label{Lemma5} For any $a>4$, any $b\in (0,1)$ and any $x\in (0,1)\setminus \{x_{min},x_{max}\}$ we have $Sf_{a,b}(x)<0$ while 
$$Sf_{a,b}(x_{min})=Sf_{a,b}(x_{max})=-\infty.$$
\end{lemma}

Recall that for $a\le 4$, the map is monotone and its dynamics are trivial.

\begin{corollary}
For any $a>4$ and any $b\in (0,1)$, the mapping $f_{a,b}$ has at most 2 minimal attracting sets.
\end{corollary}

\begin{proof} This corollary follows from \cite[Theorem 11.4]{Devaney}. Indeed, in the considered case (a mapping on a compact interval), any of those sets attracts at least one critical point.
\end{proof}

\section{Periodic solutions and bifurcations}

\subsection{The case b=1/2} 

It follows from the results of \cite[Theorem 3.10]{Misiurewicz1} that the dynamics is trivial in the case considered. In other words, two subcases are possible.
\begin{enumerate}
\item $a\le 8$, the fixed point is stable and attracts all the points of $(0,1)$.
\item $a>8$, the fixed point is unstable. There exists a period-2 orbit that attracts all the points of $(0,1)$ except the pedigree of the point $1/2$ (the point itself and all its iterative preimages).
\end{enumerate}

If $\{x_1,x_2\}$ is the period-2 orbit for some $a$ and $b=1/2$, then $x_1+x_2=1$ and, due to the symmetry of the map, we have $f'(x_1)=f'(x_2)$, and, consequently, $(f^2)'(x_1)\ge 0$. 

The period-2 point is unique and satisfies the equation 
$$\dfrac{x_1}{1-x_1}=e^{a(2x_1-1)/4}$$ 
(this follows from equation $f(x_1)=1-x_1$). Consequently
$x_1=e^{-a/4+o(a)}$ as $a\to +\infty$.

\subsection{Period-2 solution}

\begin{lemma}\label{Lemma4} For any values of parameters $a>0$ and $b\in (0,1)$ such that $a>2/b(1-b)$ the considered map admits the unique period-2 solution. 
\end{lemma}

This statement follows immediately \cite[Lemmas 1 and 3]{aks}.

Observe that for $b<1/2$ and for sufficiently large values of $a$, the period-2 orbit $\{\xi_1,\xi_2\}$ is such that $\xi_1 < x_{max}$ (this follows from the fact that $x_{max}+f_{max}>1$). Consequently, $\xi_2 \approx 2b$ and the derivative $(f^2)'(\xi_1)=(f^2)'(\xi_2)<0$. The case $b>1/2$ is symmetric.

\subsection{Bifurcations} 
The considered map always admits a fixed point. It witnesses the period-doubling bifurcation at 
$$a=a_0:=\dfrac{2}{b-b^2}.$$

Moreover, the following statement is true (see \cite{Misiurewicz1}).

\bigskip

\begin{prop}\label{Proposition2} \hfill
\begin{enumerate}
\item For any $a\in (0,a_0]$, the fixed point $b$ is asymptotically stable and attracts all the points of $(0,1)$. 
\item For any $a>a_0$, the point $b$ is repelling. There is a period-2 orbit $\{p_a,q_a\}$ with 
$$0<p_a<b<q_a<1.$$ 
This period-2 orbit is stable in the right neighborhood of $a_0$.
\end{enumerate}
\end{prop}

\begin{remark}\label{rbist} The numerical results performed in \cite[Section 4]{aks}, indicate that the period-2 orbit cannot coexist with other stable periodic orbits. However, the
so-called bistability, i.e. co-existence of two attracting orbits is rarely observed, is still possible for the considered map. For example, (as was shown numerically in \cite[Section 4]{aks}), two stable orbits of period 4 occur for $a=19.06$, $b=0.3961$; stable orbits of periods 20 and 56 coexist if $a=28.8695$, $b=0.414652$, and two chaotic narrow-band attractors can be seen for $a=28.8708$, $b=0.4141637$.   

Generally speaking, according to the mentioned numerical experiments, for any pair $(m,n)\in {\mathbb N}^2$, $m,n>2$, $(m,n)\neq (3,3)$, two distinct stable periodic solutions of periods $m$ and $n$ may coexist. 

Those numerical studies demonstrate the existence of period-doubling cascades with smooth bifurcation curves. If two curves, corresponding to distinct cascades, intersect transversally (this is observed numerically), any neighborhood of the intersection point contains two stable periodic points of distinct periods. It looks like there are infinitely many intersections of this type.

In a nutshell, bistability is a rare but possible phenomenon for the considered system. Here we notice that in the preprint \cite{Misiurewicz4} some regions of the parameters were described where the considered map may only have one attracting periodic solution.

\end{remark}

\section{Chaotic dynamics, hyperbolic chaos}

The main claim of this section states the existence of a kind of symbolic dynamics for $f_{a,b}$ (and guarantees the existence of periodic points of any periods) in the case $b\neq 1/2$. Observe that for $b=1/2$, we see either a stable fixed point $b$ that attracts all orbits except $0$, $1$, or an attracting orbit of period 2.

Let $\Sigma$ be the set of all one-sided infinite sequences of symbols '0' and '1' endowed with the standard metrics:
$$\Sigma=\{\sigma=\left(\sigma_j\in \{0,1\}:j\in {\mathbb N}\cup\{0\} \right)\};$$
$$d_\Sigma(\sigma,\theta)=\sum_{j=0}^\infty \dfrac{|\sigma_j-\theta_j|}{2^j}.$$
Consider a subset $\Sigma_0\subset \Sigma$ given by the following formula:
$$\Sigma_0=\{\sigma\in \Sigma: \sigma_j\sigma_{j+1}=0,\quad \forall j\in {\mathbb N}\cup\{0\}\}.$$
In other words, the set $\Sigma_0$ consists of segments without neighboring symbols '1'.

We list some evident properties of $\Sigma_0$.
\begin{enumerate}
\item $\Sigma_0$ is an infinite closed subset of $\Sigma$ and, hence, compact. This is because the set of all sequences of $\Sigma$ with '1' at positions $j$ and $j+1$ is open for any $j$;
\item $\Sigma_0$ is invariant with respect to the standard shift mapping $S$ on $\Sigma$ (erasing the first symbol of a sequence).
\item The periodic points of the map $S$ are dense in $\Sigma_0$. This is because the periodic points may be obtained by an infinite repetition of finite sequences of digits (if this does not result in entries '11'), e.g. $010\,010 \ldots$ or $10100\,10100 \ldots$.
\item For any $m\in {\mathbb N}$ the set $\Sigma_0$ contains at least one point of period $m$ (for instance, the point corresponding to the repeating sequence $10\ldots 0$ for $m>1$ and the sequence of zeros if $m=1$).
\item The set $\Sigma_0$ is transitive with respect to $S$.
The point with a dense orbit may be obtained by writing down all the finite admissible sequences of '0' and '1', separated by '0':
$$\sigma^*= 0\, 0\, 1\, 0\, 00\, 0\, 01\, 0\, 10\, 0\, 000\, 0\, 001\,0 \, 010\, 0\, 100\, 0\, 101\, 0 \ldots$$
\end{enumerate}

\begin{theorem}\label{Theorem1} For any $b\in (0,1)$, $b\neq 1/2$, there is $a_0=a_0(b)>0$ such that for any $a>a_0$ there exists the compact set $Q\subset [f_{min},f_{max}]$, invariant with respect to $f_{a,b}$ and such that
\begin{enumerate}
\item $|(f^2_{a,b})'(x)|>1$ for any $x\in Q$;
\item for any $m\ge 2$, the mapping $f^m_{a,b}|_Q$ is topologically conjugated to the Bernoulli shift on the set $\Sigma_0$. 
\end{enumerate}
\end{theorem}

The proof of Theorem \ref{Theorem1} is given in Section 5.

\section{Proof of Theorem \ref{Theorem1}}

\subsection{Transformation of coordinates}\label{subsec:conjugacy}
Introduce a family of maps $g_{a,b}:\mathbb{R}\to \mathbb{R}$ by 
\begin{equation}\label{equ:smoothconjugate}
g_{a,b}(y):= y + \dfrac{a}{e^y+1}-ab.
\end{equation}
with $a>0$ and $b\in(0,1)$. See Figure \ref{fig_2} for the graph of $g_{30,1/3}$. This family of maps $\{g_{a,b}\}$ has been introduced in \cite{EOS96}, 
and already been studied, see \cite{aks} and references therein. It was shown in \cite{Misiurewicz4} that for each $a,b$ by taking 
$$
y=h(x)=\ln\frac{1-x}{x}, \qquad 
\left( x=\dfrac1{e^y+1} \right)
$$
we have 
\begin{equation}\label{equ:smoothconjugate2}
g_{a,b}(y)=h\circ f_{a,b}\circ h^{-1}(y).
\end{equation}
In other words, the following diagram commutes
\[
\begin{CD}
(0,1) @>f_{a,b}(x)>> (0,1) \\
@VhVV @VhVV \\
\mathbb{R} @>g_{a,b}(y)>> \mathbb{R}
\end{CD}
\]
Therefore, $f_{a,b}$ and $g_{a,b}$ are smoothly conjugated. 

Consequently, the results of the papers \cite{aks} and \cite{Misiurewicz1} can be merged. For example, one can classify the absorbing intervals of the mapping $f_{a,b}$ as it was done in \cite{aks} for $g_{a,b}$. 

\subsection{Key lemma for $g_{a,b}$}\label{subsection:keylemma}
In the view of \eqref{equ:smoothconjugate2}, to prove Theorem \ref{Theorem1}, it suffices to prove the following Key lemma for the function $g_{a,b}$ given by \eqref{equ:smoothconjugate}. 

\begin{figure}[t!]\centering
  \includegraphics[width=.84\textwidth]{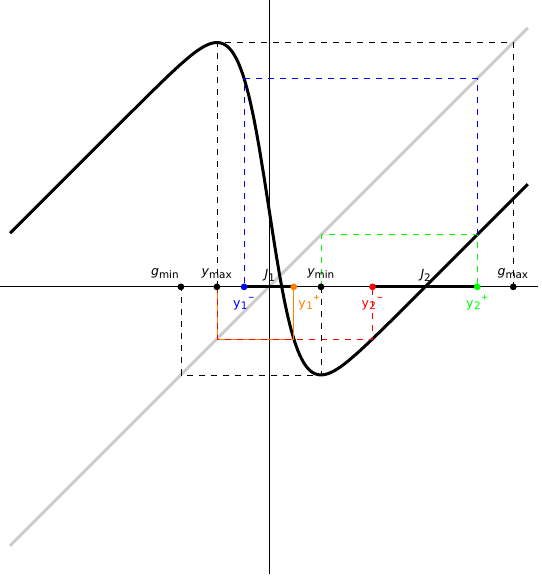}
\caption{\footnotesize The graph of the map $g_{30,1/3}$.}
\label{fig_2}
\end{figure}

\begin{lemma}[Key Lemma]\label{ta1} For any $b\in (0,1)$, $b\neq 1/2$, there is $a_0=a_0(b)>0$ such that for any $a\geq a_0$ there exists the compact set $K\subset {\mathbb R}$, invariant with respect to $g_{a,b}$ and such that
\begin{enumerate}
\item $|(g^2_{a,b})'(x)|>1$ for any $x\in K$;
\item the mapping $g_{a,b}|_K$ is topologically conjugated to the Bernoulli shift of the set $\Sigma_0$. 
\end{enumerate}
\end{lemma}

\begin{proof}
We assume, without loss of generality, that $b<1/2$, we always suppose this value to be fixed. 

The map $g_{a,b}$ has the unique fixed point $y_0=\ln((1-b)/b)$.

The derivative of $g_{a,b}$ is given by the formula 
$$g'_{a,b}(y)=1-\dfrac{ae^y}{(e^y+1)^2}.$$
The critical points are 
$$y_{max}=\ln \left(\dfrac{a}2-1-\sqrt{\dfrac{a^2}4-a}\right)$$
and
$$y_{min}=\ln \left(\dfrac{a}2-1+\sqrt{\dfrac{a^2}4-a}\right).$$
Observe that $y_{min}+y_{max}=0$ and 
$$e^{y_{min}}=\dfrac{a}2-1+\sqrt{\dfrac{a^2}4-a}=a+o(a)$$
as $a\to+\infty$. 
The function $g_{a,b}$ increases on $(-\infty,y_{max}]$ and $[y_{min},+\infty)$ and decreases on $[y_{max},y_{min}]$.

In our assumptions, the following inequalities hold:
$$y_{max}<0<y_0<y_{min}.$$

The second derivative of $g_{a,b}$ may be calculated by the formula:
\begin{equation}\label{equ:secondderivative}
g''_{a,b}(y)=\dfrac{ae^y(e^y-1)}{(e^y+1)^3}.
\end{equation}
Evidently, this function is positive for positive values of $y$ and vice-versa. The shape of $g_{a,b}$ is as illustrated in Figure \ref{fig_2}.

Denote 
$$
g_{min}:=g_{a,b}(y_{min})\qquad g_{max}:=g_{a,b}(y_{max}),
$$
then we have 
\begin{lemma}\label{la1} The value $a_0$ can be selected sufficiently large, so that 
\begin{equation}\label{eqa1}
g_{\min}<y_{\max}, \qquad g_{\max}>y_{\min},\qquad g_{a,b}(g_{\max})>y_{\min}.
\end{equation}
for any $a\ge a_0$.
\end{lemma}

\begin{proof}[Proof of Lemma \ref{la1}] Let us verify Lemma \ref{la1} item by item. The first inequality of \eqref{eqa1} can be rewritten as follows
$$
\ln\left(\dfrac{a}2-1+\sqrt{\dfrac{a^2}4-a}\right)+
\dfrac{a}{\frac{a}2+\sqrt{\frac{a^2}4-a}}-ab<\ln\left(\dfrac{a}2-1-\sqrt{\dfrac{a^2}4-a}\right)
$$
or
\begin{equation}\label{eqa11}
2\ln\left(\dfrac{a}2-1+\sqrt{\dfrac{a^2}4-a}\right)+
\dfrac{a}{\frac{a}2+\sqrt{\frac{a^2}4-a}}-ab<0
\end{equation}
Here we took into account the fact that $y_{min}+y_{max}=0$. The first term of the left-hand side \eqref{eqa11} equals $2\ln a + o(1)$, the second one is $1+o(1)$, and the third one is $-ab$. All the expression is negative for sufficiently big $a$ provided $b$ is positive.

The second inequality of \eqref{eqa1} looks as follows:
$$\ln\left(\dfrac{a}2-1-\sqrt{\dfrac{a^2}4-a}\right)+
\dfrac{a}{\frac{a}2-\sqrt{\frac{a^2}4-a}}-ab>\ln\left(\dfrac{a}2-1+\sqrt{\dfrac{a^2}4-a}\right)
$$
\begin{equation}\label{eqa>0}
\dfrac{a}{\frac{a}2-\sqrt{\frac{a^2}4-a}}-ab>2\ln\left(\dfrac{a}2-1+\sqrt{\dfrac{a^2}4-a}\right)
\end{equation}
that is $a-ab+o(a)>2\ln (a+o(a))$. Here we took into account the fact that
$$\dfrac{a}{\frac{a}2-\sqrt{\frac{a^2}4-a}}=\frac{a}2+\sqrt{\frac{a^2}4-a}=a+o(a).$$
The inequality \eqref{eqa>0} is obviously true for big values of $a$ since $b<1$.

Finally, the last inequalities of \eqref{eqa1} can be expressed in the form
$$g_{max}+\dfrac{a}{e^{g_{max}}+1}-ab>y_{min}$$
or
$$y_{max}+\dfrac{a}{e^{y_{max}}+1}+\dfrac{a}{e^{g_{max}}+1}-2ab>y_{min}$$
or
$$2y_{max}+\dfrac{a}{e^{y_{max}}+1}+\dfrac{a}{e^{g_{max}}+1}-2ab>0$$
$$\begin{array}{c}
2 \ln \left({\dfrac{a}2-1-\sqrt{\dfrac{a^2}4-a}}\right)+\dfrac{a}{\dfrac{a}2-\sqrt{\dfrac{a^2}4-a}}+\\
\dfrac{a}{1+\left(\dfrac{a}{2}-1-\sqrt{\dfrac{a^2}{4}-a}\right) \exp\left(\dfrac{a}{a/2-\sqrt{a^2/4-a}}-ab\right)}-2ab>0.
\end{array}$$
The first summand on the last formula is $o(a)$, the second is $a+o(a)$, the third one tends to 0 and the last one is $-2ab$. Therefore, the overall sum is $a(1-2b)+o(a)$ which must be positive. We complete the proof of Lemma \ref{eqa1}.
\end{proof}

\medskip

As it follows from the statement of the Lemma, $g_{min}<y_{max}<y_{min}$ and the map $g_{a,b}$  increases on $(y_{min},+\infty)$.
Consequently (as well as plotted in Figure \ref{fig_2}), there is a unique point $y_2^+$ in $[y_{min},+\infty)$ such that $g_{a,b}(y_2^+)=y_{min}$. 

Now, we prove that $y_2^+\in g_{a,b}[y_{max},y_{min}]$. Evidently, $g_{min}=g_{a,b}(y_{min})<x_{min}<y_2^+$. 

Now let us prove that $y_2\le g_{max}$. We take into account the inequality $g_{max}>y_{min}$ and the fact that $g_{a,b}$ increases for $y>y_{min}$. So, if there were $g_{max} < y_2^+$, we would have $g(g_{max})< g(y_2^+)=y_{min}$, which contradicts the last of inequalities \eqref{eqa1}.

Then there exists a unique point $y_1^-\in [y_{max},y_{min}]$ such that $g_{a,b}(y_1^-)=y_2^+$.

Besides, $y_{max}\in [g_{min},g_{max}]$ for big values of $a$. So, there are unique points $y_1^+\in [y_{max},y_{min}]$ and $y_2^-\in [y_{min},+\infty)$ such that 
$$g_{a,b}(y_1^+)=g_{a,b}(y_2^-)=y_{max}.$$

Now, we consider two segments: 
$$
J_1:=[y_1^-,y_1^+]\subset (y_{max},y_{min})~~~ \mbox{and}~~~~
J_2:=[y_2^-,y_2^+]\subset (y_{min},+\infty).
$$

These two segments are disjoint. Moreover, we have 
$$
g_{a,b}(J_1)=[y_{\max},y_2^+]\supset J_1\cup J_2, ~~~  g_{a,b}(J_2)=[y_{\max},y_{min}]\supset J_1, ~~~ \mbox{and}~~~
g_{a,b}(J_2)\cap J_2=\emptyset$$
These facts immediately imply that 
$$g_{a,b}^2(J_i)\supset J_1\cup J_2, \qquad i=1,2.$$
Applying $g_{a,b}$ to the latter inclusion, we can easily see that 
$$g_{a,b}^m(J_i)\supset J_1\cup J_2, \qquad i=1,2$$
for any $m\ge 2$. 
We set 
\begin{equation}\label{equ:horseshoe}
K:=\bigcap_{k=0}^{\infty} g_{a,b}^{-k} (J_1\cup J_2).
\end{equation}

This $K$ in \eqref{equ:horseshoe} is a nonempty compact invariant set. Given a point $y\in K$, we construct a sequence $\eta(y)=(\sigma_j: j\in {\mathbb N}\cup\{0\})\in \Sigma$ as follows:
$\sigma_m=\epsilon\in \{0,1\}$ if $g^m_{a,b}(y)\in J_{\epsilon+1}$.

Since $g_{a,b}(J_2)\cap J_2 =\emptyset$, we have  
$\eta(y)\in \Sigma_0$ for any $y\in K$.

Now we prove that for any point of $\Sigma_0$ there exists a corresponding point of $K$.

\begin{lemma}\label{la2} Given a sequence $\Bar{\sigma}\in \Sigma_0$, there exists a point $\Bar{y}\in \Sigma_0$ such that
$$\eta(\Bar{y})=\Bar{\sigma}.$$
\end{lemma}

\begin{proof} Let $\Bar{\sigma}=(\sigma_j:j\in {\mathbb N}\cup\{0\})$. We shall construct a nested sequence of segments 
$$I_0\supset I_1 \supset I_2\supset \ldots$$
First of all, we take $I_0=J_{\sigma_0+1}$. 

Let the segment $I_m$ be already constructed so that 
$g_{a,b}^m(I_m)=J_{\sigma_{m}+1}$. Since $\Bar{\sigma}\in \Sigma_0$, we have $g_{a,b}(J_{\sigma_m+1})\supset J_{\sigma_{m+1}+1}$. So, we can take a subsegment $I_{m+1}\subset I_m$ so that $g_{a,b}^m(I_{m+1})=J_{\sigma_{m+1}+1}$. 

We repeat this procedure infinitely many times. So, we can take any point ${\Bar{y}}\in \bigcap_{m=0}^\infty I_m$.
\end{proof}

To prove that the mapping $\eta$ is one-to-one, we demonstrate that the mapping $g_{a,b}$ is uniformly expanding on each of the sequences $J_i$.
So, it suffices to estimate the derivatives of $g_{a,b}^2$.

\begin{lemma}\label{la2} The value $a_0$ can be taken so big that
\begin{equation}\label{equ:hyperbolicity}
\min\left\{|g'_{a,b}(y_1^-)|,|g'_{a,b}(y_1^+)|\right\}\cdot g'_{a,b}(y_2^-)>1
\end{equation}
\end{lemma}

\begin{proof} 
Firstly, we give an estimate for $g'_{a,b}(y_2^-)$.
We have 
\begin{equation}\label{equ:gab}
g_{a,b}(y_2^-)=y_2^-+\dfrac{a}{e^{y_2^-}+1}-ab=y_{max}.
\end{equation}
On the other hand, $y_2^->y_{min}$, hence
$$0<\dfrac{a}{e^{y_2^-}+1}<\dfrac{a}{e^{y_{min}}+1}<2$$
for sufficiently big $a$.

Therefore, using \eqref{equ:gab} and the last inequality, we get $e^{y_2^-}>e^{y_{max}} e^{ab-2}$, 
$$\dfrac{ae^{y_2^-}}{(e^{y_2^-}+1)^2}<ae^{-y_2^-}<ae^{-y_{max}} e^{2-ab}=ae^{y_{min}} e^{2-ab}
<a^2e^{-ab+3}$$
for big values of $a$, and 
\begin{equation}\label{eqest1}
 g'_{a,b}(y_2^-)>1-a^2 e^{-ab+3}   
\end{equation}
for big values of $a$.

Secondly, we have $g_{a,b}(y_1^+)=y_{max}$ that is
$$y_1^++\dfrac{a}{e^{y_1^+}+1}-ab=y_{max}$$
which means that 
$$\dfrac{a}{e^{y_1^+}+1}=y_{max}-y_1^++ab$$
and, since $y_1^+\in (y_{max},y_{min})$, we have 
$$\dfrac{a}{e^{y_1^+}+1}=ab+o(a)$$
and, consequently
$$\dfrac{1}{e^{y_1^+}+1}=b+o(1), \qquad e^{y_1^+}=\dfrac1{b}-1+o(1).$$

Therefore,
\begin{equation}\label{eqest2}
 g'_{a,b}(y_1^+)=1-ab(1-b)+o(a)   
\end{equation}
which tends to $-\infty$ as $a\to \infty$.

Finally, we have $g_{a,b}(y_1^-)=y_2^+$ and $g_{a,b}(y_2^+)=y_{min}$. The last inequality together with the definition of $g_{a,b}$ and the fact that $y_2^+>y_{min}$ imply that $y_2^+=ab+o(a)$. This means that 
$$y_1^-+\dfrac{a}{e^{y_1^-}+1}-ab =ab+o(a).$$
Since $y_1^-\in (y_{max},y_{min})$, we have
$$\dfrac{a}{e^{y_1^-}+1}=2ab+o(a),$$
$$\dfrac{1}{e^{y_1^-}+1}=2b+o(1), \qquad e^{y_1^-}=\dfrac1{2b}-1+o(1).$$

This means that 
\begin{equation}\label{eqest3}
 g'_{a,b}(y_1^{-})=1-2ab(1-2b)+o(a)   
\end{equation}
which also tends to $-\infty$ for big values of $a$.

Therefore, \eqref{equ:hyperbolicity} directly follows from inequalities \eqref{eqest1}, \eqref{eqest2}, and \eqref{eqest3}. Thus, we complete the proof of Lemma \ref{la2}.
\end{proof}

Lemma \ref{la2}, together with the properties of the second derivative of $g_{a,b}$ (see \eqref{equ:secondderivative}) and the fact that for any $x\in K$
$$ (x\in J_2)\implies (g_{a,b}(x) \in J_1)$$
implies that 
\begin{equation}\label{equ:hyperbolicity2}
(g^2_{a,b})'(x)>1 \qquad \forall x\in K.
\end{equation}
Thus the conjugacy between $g_{a,b}^m|_K$ and the Bernoulli shift on $\Sigma_0$ follows from \eqref{equ:hyperbolicity2} and the disjointedness of $J_1$ and $J_{2}$. Thus, Key Lemma \ref{ta1} is proven.
\end{proof}

\begin{rem}
The theorem implies, in particular, the existence of a periodic hyperbolic orbit of any given period for big values of $a$. The existence without hyperbolicity follows from the results of \cite{Misiurewicz1}. However, the techniques of the proof of that paper, based on the result of \cite{Misiurewicz2}, give no chance to find out if the obtained periodic points (and the chaotic invariant set) are hyperbolic. \end{rem}

\section{One-dimensional maps with fixed mean values of periodic orbits}

Define a specific class of one-dimensional systems.

\begin{definition}\label{dm}
Let $I$ be a convex subset of $\mathbb R$. We say that a measurable mapping $f:I \mapsto I$ is of the class \emph{$\mathcal M$} if there exists a constant $b\in {\mathbb R}$ such that for any Borel probability $f$-invariant measure $\mu$ we have 
\begin{equation}\label{eqb}
\int_{I} x \, d\mu =b.
\end{equation}
In particular, for any $n$-periodic point $x_0$ (i.e. $f^n(x_0)=x_0$), we have
\begin{equation}\label{eqb1}
\dfrac1{n} \sum_{j=0}^{n-1} f^j(x_0) =b.
\end{equation}
If we need to specify the constant $b$ in conditions \eqref{eqb} and \eqref{eqb1}, we say that the mapping $f\in {\mathcal M}_b$.
\end{definition}

In \cite[Corollary 3.5]{Misiurewicz1} it was 
proven that $f_{a,b}\in {\mathcal M}_b$ 
where $f_{a,b}$ is the replicator map considered in previous sections. 
This solves a known question from \cite{Misiurewicz3}: whether there are any smooth maps of $\mathcal M$, with many periodic points other than 
\begin{equation}\label{aeminusx}
x\mapsto Axe^{-x}.
\end{equation}
Besides, an open question was formulated: to find all nontrivial smooth (analytic) maps having infinitely many periodic orbits for which centers of mass of all periodic orbits coincide.

We are still quite far from being sure to construct all the possible maps of $\mathcal M$. However, in this section, we describe a very broad class of such mappings.

Let us start with a very trivial observation. 

\begin{prop} Let $f:I \to I$ be a measurable mapping of a convex subset of $\mathbb R$. Then, for any probability $f$-invariant measure $\mu$ and any function $\varphi\in {\mathbb L^1}(\mu)$ we have $$\int_I (\varphi(f(x))-\varphi(x))\, d\mu=0$$  
\end{prop}

So, we get the following statement.

\begin{theorem}
Let $f:{\mathbb R} \to {\mathbb R}$ be a continuous mapping of a convex subset of a line. If there is a piecewise continuous function $H: I \mapsto {\mathbb R}$ and a constant $b\in {\mathbb R}$ such that 
\begin{equation}\label{eqh}
H(f(x))-H(x)=x-b
\end{equation}
then $f\in {\mathcal M}_b$.
\end{theorem}

\begin{corollary}
    Let $g(y)=y+\varphi(y)$ where $\varphi$ is a continuous strictly decreasing map. Then the transformation $x=\varphi(y)$ reduces the map $g$ to that of the class ${\mathcal M}_0$. Here we have $\varphi=H^{-1}$.
\end{corollary}

Equation \eqref{eqh} implies that for any $n\in {\mathbb N}$ we have 
$$H(f^n(x))-H(x)=x+f(x)+\ldots+f^{n-1}(x)-nb.$$

The latter equation implies that the Birkhoff averages  
$$\lim_{n\to \infty} \dfrac{1}{n} (x+f(x)+\ldots+f^{n-1}(x))=b$$
for all $x$ such that $H(f^n(x))$ is bounded (or, at least, $H(f^n(x))=o(n)$).

\begin{rem}
    Observe that if the mapping $H$ is invertible, equation \eqref{eqh} implies that
    any map  
    \begin{equation}\label{eqgg}
    f(x)=H^{-1}(H(x)+x-b)\in {\mathcal M}_b.
    \end{equation}
    If 
    $$H(x)=\dfrac1{a}\left(\ln\dfrac{1-x}{x}\right)=\dfrac{h(x)}{a},$$ the formula \eqref{eqgg} gives us the mapping $f_{a,b}$ introduced above; for $H(x)=-\ln x$, we get the mapping \eqref{aeminusx} with $A=e^b$. However, equation \eqref{eqgg} can generate multiple examples of mappings from ${\mathcal M}_b$ (also having some other similar properties). For example, one can consider the family $$\arctan (\tan x - a(x-b))$$
    or take $H=\Phi^{-1}$ where $\Phi$ is the cumulative function of the normal distribution.
\end{rem}

\begin{rem} 
The fact that $f\in {\mathcal M}$ does not imply that it has infinitely many periodic points. We know this fact for the map $f_{a,b}$ introduced above. Meanwhile, Theorem \ref{Theorem1} guarantees that the map $f_{a,b}$ has a hyperbolic chaotic invariant set for some values $a$ and $b$. Suppose that it is so for $a=a_0$, $b=b_0$. Then we can take  $${\tilde f}_H=H^{-1}(H(x)+x-b_0)$$
for some $H$, $C^1$ close to $\frac1{a_0}(\ln(1-x)/x)$ in the absorbing segment of $f_{a_0,b_0}$ (actually, this is $[f_{max},f_{min}]$). We get a map of the class $\mathcal M$ that still has a chaotic invariant set and, consequently, infinitely many periodic points. 
\end{rem}

\begin{rem} Finally, the following open question appears: is there is any map of the class $\mathcal M$ that cannot be described by Equation \eqref{eqgg} on its nonwandering set?
\end{rem}

\section*{Acknowledgments}
S. K. would like to thank Peking University for their support and hospitality. Y. Z. would like to thank the Mathematical Institute of the Polish Academy of Sciences for their support and hospitality, and is partially supported by NSFC Nos. 1216114100212271432 and USTC-AUST Math Basic Discipline Research Center. The authors would also like to thank the VIII Symposium on Nonlinear Analysis in Toru\'{n} 2024, where some parts of this work were done.

\end{document}